\documentclass[12pt, a4paper]{article}

\usepackage{indentfirst}
\usepackage[leqno]{amsmath}
\usepackage{amssymb}
\usepackage{amsthm}
\usepackage{enumerate}
\usepackage{cases}
\usepackage[center]{titlesec}
\usepackage{fancyhdr}
\usepackage[colorlinks, 
linkcolor=blue,
anchorcolor=blue, 
citecolor=red]{hyperref}

\numberwithin{equation}{section} 
\newtheorem{lem}{Lemma}[section]
\newtheorem{thm}[lem]{Theorem}

\newtheorem{cor}[lem]{Corollary}

\title{Sobolev Inequality In Manifolds With Lower Quadratic Curvature Decay\thanks{\textbf{Mathematics Subject Classification} ~~35R45; 53C21}}
\date{}

\begin{document}
	\author{Tian Chong\thanks{Supported by the National Natural Science Foundation of China (NSFC) No. 12001360}, Han Luo\thanks{supported by Jiangsu Funding Program for Excellent Postdoctoral Talent (2024ZB650) and the Fundamental Research Funds for the Central Universities (30924010314)}, Lingen Lu\thanks{Corresponding author}}
	 \maketitle

	\pagestyle{plain}
	
	\begin{abstract}
		ABSTRACT. 
		By using the ABP method developed by Cabr\'e and Brendle, we establish some Sobolev inequalities for compact domains and submanifolds in a complete Riemannian manifold with lower quadratic curvature decay.
	\end{abstract}
	
\section{Introduction}

How curvature affects functional inequalities on manifolds is an interesting question in geometry. As one of the most important functional inequalities, Sobolev inequalities relate $L^p$ norms of functions to $L^q$ norms of their gradients. In the early 1970s, a Sobolev inequality for submanifolds of Euclidean space was established independently by Michael and Simon \cite{MS} and by Allard \cite{Al}. Recently, inspired by the ABP method developed by Cabr\'e \cite{Ca}, Brendle \cite{Br1} proved a Sobolev inequality of submanifolds in Euclidean space which implies a sharp isoperimetric inequality for minimal submanifolds in Euclidean space of codimension at most $2$. Later, he \cite{Br2} also generalized the Euclidean case to the manifold with nonnegative curvature. For more about the Sobolev inequality on manifolds, we refer the readers to \cite{BK,Hebey, Jo, MW} and references therein.

The goal of this paper is to derive some Sobolev inequalities for compact domains and submanifolds in a complete Riemannian manifold with lower quadratic curvature decay. In \cite{Lo, Y}, the authors study the finite topological type property of manifold with lower quadratic curvature decay. A striking result of Lemma 2.1 in \cite{LS} implies that there exists a complete metric of quadratic curvature decay for any noncompact manifold. Let $\lambda(t):[0,+\infty)\to[0,+\infty)$ be a nonnegative and nonincreasing continuous function satisfies
\begin{equation}\label{G}
		B:=\varlimsup_{t\to\infty}t^2\lambda(t)<\infty.
\end{equation}
Obviously, the existence of $B$ implies 
\begin{equation}\label{G1}
	b_1:=\int_0^\infty \lambda(t)dt<\infty.
\end{equation}
Set $h_1(t),h_2(t)$ be the unique solutions of
\begin{equation}
\left\{
\begin{aligned}
	&h_1''(t)=\lambda(t)h_1(t),\\
	&h_1(0)=0, h_1'(0)=1,
\end{aligned}
\right.
\quad
\left\{
\begin{aligned}
	&h_2''(t)=\lambda(t)h_2(t),\\
	&h_2(0)=1, h_2'(0)=0.
\end{aligned}
\right.
\label{h}
\end{equation}
By the ODE theory, we know that $h_1(t),h_2(t)$ are exist for all $t\in[0,\infty)$. 

We define a complete noncompact Riemannian manifold $M$ of dimension $n$ with \textbf{lower quadratic Ricci curvature decay} if there is a base point $o\in M$ such that
\begin{equation}\label{Ric}
	\mathrm{Ric}(q)\geq -(n-1)\lambda(d(o,q)),\quad\forall q\in M,
\end{equation}
where $d$ is the distance function of $M$. Then we define the \textbf{asymptotic volume ratio} of $M$ as
$$
\theta:=\lim_{r\to\infty}\frac{|\{q\in M:d(o,q)<r\}|}
{n|B^n|\int_0^rh_1^{n-1}(t)dt},
$$
where $o$ is the base point, $B^n$ is the unit ball in $\mathbb{R}^n$. By Theorem 2.14 in  \cite{PRS}, one can show that
$$
\frac{|\{q\in M:d(o,q)<r\}|}
{n|B^n|\int_0^rh_1^{n-1}(t)dt}
$$
is a nonincreasing function and $\theta\leq 1$. 

Through combining the ABP method developed in \cite{Br2} and the skill of taking out the direction of motion in \cite{Vi} with some comparison theorems, we obtain the following Sobolev type inequality for a compact domain in a Riemannian manifold that has lower quadratic Ricci curvature decay.
	
\begin{thm}\label{thm1.1}
	Let $M$ be a complete noncompact manifold of dimension $n$ with lower quadratic Ricci curvature decay with respect to a base point $o\in M$. Let $\Omega$ be a compact domain in $M$ with boundary $\partial\Omega$ and $f$ be a positive smooth function on $\Omega$. Then
	$$
	\int_{\partial \Omega} f+\int_\Omega |D f|+2(n-1)b_1\int_\Omega f
	\geq n\Big(\frac{|B^n|\theta}{\frac{1+\sqrt{1+4B}}{2}(2e^{r_0b_1}-1)^{n-1}}\Big)^\frac{1}{n}
	\Big(\int_\Omega f^\frac{n}{n-1}\Big)^\frac{n-1}{n},
	$$
	where $r_0=\max\{{d}(o,x)|x\in\Omega\}$, $D$ is Levi-Civita connection of $M$, $\theta$ is the asymptotic volume ratio of $M$ and $B,b_1$ are defined in \eqref{G} and \eqref{G1}.
\end{thm}	

When $b_1=0$, Theorem \ref{thm1.1} implies Theorem 1.1 of \cite{Br2}. Let $f=1$ in Theorem \ref{thm1.1}, we obtain a isoperimetric inequality for manifold with lower quadratic Ricci curvature decay.
\begin{cor}
	Let $M,\Omega, r_0,\theta,B,b_1$ be defined as the same way in Theorem \ref{thm1.1}. Then
$$
|\partial \Omega|\geq \Big(n\Big(\frac{|B^n|\theta}{\frac{1+\sqrt{1+4B}}{2}(2e^{r_0b_1}-1)^{n-1}}\Big)^\frac{1}{n}
-
2(n-1)b_1|\Omega|^{\frac{1}{n}}\Big)|\Omega|^{\frac{n-1}{n}}.
$$	
\end{cor}

Next, we consider Sobolev inequalities for submanifolds. We define a complete noncompact Riemannian manifold $M$ of dimension $n+p$ with \textbf{lower quadratic sectional curvature decay}, if there is a base point $o\in M$ such that
\begin{equation}\label{SG}
\bar{R}(e_B, e_C, e_B ,e_C)(q)\geq -\lambda(d(o,q)),\quad \forall q\in M,1\leq B,C\leq n+p,
\end{equation}
where $\lambda$ is the function defined in $(\ref{G})$, $\bar{R}$ is the curvature tensor of $M$ and $\{e_{A}\}_{1\leq A\leq n+p}$ is a orthonormal basis of $M$ at point $q$.

Similarly, we establish a Sobolev type inequality for a compact submanifold in a Riemannian manifold with lower quadratic sectional curvature decay.

\begin{thm}\label{thm1.3}
Let $M$ be a complete noncompact manifold of dimension $n+p$ with lower quadratic sectional curvature decay respect to a base point $o\in M$. Let $\Sigma$ be a compact submanifold of $M$ of dimension $n$ (possibly with boundary $\partial\Sigma$), and let $f$ be a positive smooth function on $\Sigma$. If $p\geq2$, then
$$
\begin{aligned}
	&\int_{\partial\Sigma}f+\int_\Sigma\sqrt{|D f|^2+f^2|H|^2}
	+(2nb_1+1)\int_\Sigma f\\
	&\geq
	n\Big(\frac{2b_1(n+p)|B^{n+p}|\theta}{(e^{2b_1}-1)p|B^p|(2e^{r_0b_1}-1)^{n+p-1}\frac{1+\sqrt{1+4B}}{2}}\Big)^{\frac{1}{n}}
	\Big(\int_\Sigma f^{\frac{n}{n-1}}\Big) ^{\frac{n-1}{n}}.
\end{aligned}
$$
where $r_0=\max\{{dist}(o,x)|x\in\Sigma\}$, $H$ is the mean curvature vector of $\Sigma$, $D$ is Levi-Civita connection of $\Sigma$, $\theta$ is the asymptotic volume ratio of $M$ and $B,b_1$ are defined in \eqref{G} and \eqref{G1}.
\end{thm}

\section{Proof of Theorem \ref{thm1.1}}	
In this section, we assume that $M^n$ is a complete noncompact Riemannian manifold whose Ricci curvature satisfies \eqref{Ric} with respect to a base point $o\in M$. Let $\Omega$ be a compact domain in $M$ with smooth boundary $\partial\Omega$ and $f$ be a smooth positive function on $\Omega$.

It suffices to treat the case that $\Omega$ is connected. By using the scaling skill, one can assume that
\begin{equation}
	\label{scale-d}
	\int_{\partial\Omega} f+\int_\Omega|D f|+2(n-1)b_1\int_\Omega f=n\int_\Omega f^\frac{n}{n-1}.
\end{equation}
Since $\Omega$ is connected, we can find a solution of the following Neumann problem
\begin{equation}
	\left\{
	\begin{array}{lr}
		\mathrm{div}(fD u)=nf^\frac{n}{n-1}-2(n-1)b_1f-|D f|,& \text{ in }\Omega,\\
		\\
		\langle Du,\nu\rangle=1,
		& \text{ on }\partial\Omega,
	\end{array}
	\right.\label{neumann}
\end{equation}
where $\nu$ is the outward unit normal vector field of $\partial\Omega$. The standard elliptic regularity theory (see Theorem 6.31 in \cite{GT}) implies $u\in C^{2,\gamma}$, for each $0<\gamma<1$.

Following the notions in \cite{Br2}, we set
$$
U:=\{x\in\Omega\setminus\partial\Omega:0<|Du(x)|<1\}.
$$
For any $r>0$, we define $A_r$ as following:
$$
A_r:=\{\bar{x}\in U| ru(x)+\frac{1}{2}{d}
(x,\exp_{\bar{x}}(rD u(\bar{x})))^2\geq
ru(\bar{x})+\frac{1}{2}r^2|D u(\bar{x})|^2,
\forall x\in\Omega\}.
$$
For each $r>0$, we define the transport map $\Phi_r:\Omega\to M$ as
$$
\Phi_r(x)=\exp_x(rD u(x)), \quad\forall x\in\Omega.
$$
Using the regularity of the solution $u$ of the Neumann problem, we known that the transport map is of class $C^{1,\gamma},0<\gamma<1$.

\begin{lem}\label{hessu}
	If $x\in U$, we have
	\begin{equation*}
		\frac{1}{n}\Delta u\leq f^\frac{1}{n-1}-\Big(\frac{n-1}{n}\Big)2b_1.
	\end{equation*}

	\begin{proof}
		Because $x\in U$, we have $|Du|<1$. Using the Cauchy-Schwarz inequality, we get
		\begin{equation*}
			-\langle Df,Du\rangle\leq|Df|.
		\end{equation*}
	Using \eqref{neumann}, one can obtain that
	\begin{equation*}
		\begin{aligned}
		f\Delta u&=nf^\frac{n}{n-1}-2(n-1)b_1f-|Df|
		-\langle Df,Du\rangle\\
		&\leq nf^\frac{n}{n-1}-2(n-1)b_1f.
		\end{aligned}
	\end{equation*}
	Dividing by $nf>0$, we complete the proof.
	\end{proof}
\end{lem}
The proofs of the following three lemmas are identical to those for Lemmas 2.2-2.4 in \cite{Br2} for our Ricci setting. 
\begin{lem}
	The set
	\[
	\{q\in M|d(x,q)<r, \forall x\in\Omega\}\setminus\Omega
	\]
	is contained in $\Phi_r(A_r)$.
\end{lem}	
\begin{lem}\label{d-mat}
	Assume that $\bar{x}\in A_r$, and let 
	$\bar{\gamma}(t):=\exp_{\bar{x}}(tDu(\bar{x}))$ for all $t\in[0,r]$. 
	If $Z$ is a smooth vector field along $\bar{\gamma}$ satisfying
	$Z(r)=0$, then
	\[
	(D^2u)(Z(0),Z(0))+\int_0^r\big(|D_tZ(t)|^2-R(\bar{\gamma}'(t),Z(t),\bar{\gamma}'(t),Z(t))\big)dt\geq0.
	\]	
\end{lem}


\begin{lem}\label{jacovani}
Assume that $\bar{x}\in A_r$, and let 
$\bar{\gamma}(t):=\exp_{\bar{x}}(tD u(\bar{x}))$ for all $t\in[0,r]$. Moreover, 
let $\{e_1,\dots,e_n\}$ be an orthonormal basis of $T_{\bar{x}}M$. Suppose that $W$ 
is a Jacobi field along $\bar{\gamma}$ satisfying 	
\[
\langle D_tW(0),e_j\rangle=(D^2u)(W(0),e_j),\quad 1\leq j\leq n.
\]
If $W(\tau)=0$ for some $\tau\in(0,r)$, then $W$ vanishes identically.
\end{lem}

Now, we derive some comparison type results for later use.

\begin{lem}
	Let $r\in[0,r_0]$, where $r_0=\max\{{d}(o,x)|x\in\Omega\}$. Define $f(t)$ be a solution of the following problem
	\begin{equation}\label{f}
		\left\{
		\begin{aligned}
			&f''(t)=\lambda(|t-r|)f(t),\quad t\in(0,\infty),\\
			&f(0)=0, f'(0)=1.
		\end{aligned}
		\right.
	\end{equation}	
	Then we have
	\begin{equation}\label{fineq}
		f(t)\leq \frac{e^{r_0b_1}-1}{b_1}h_2(t)+e^{r_0b_1}h_1(t),\quad t\in[0,\infty).
	\end{equation}	
\end{lem}
\begin{proof}
	It is obviously that $h_1,h_2,h'_1,h'_2,f,f'$ are nonnegative and nondecreasing function on $[0,\infty)$. It is easy to show that
	$$
	(\int_0^te^{\int_0^s\tau\lambda(|\tau-r|)d\tau}ds)''\geq\lambda(|t-r|)\int_0^te^{\int_0^s\tau\lambda(|\tau-r|)d\tau}ds.
	$$
	By Lemma 2.1 in \cite{PRS}, we have
	$$
	f(r)\leq\int_0^re^{\int_0^s\tau\lambda(r-\tau)d\tau}ds\leq\int_0^re^{sb_1}ds\leq
	\frac{e^{r_0b_1}-1}{b_1}.
	$$
	For it's derivative, we have
	$$
	f'(r)= 1+\int_0^r \lambda f(t)dt\leq e^{r_0b_1}
	$$
	When $t\geq r$, we have the following ODE
	$$
	f''(t)=\lambda(t-r)f(t).
	$$
	with initial value $f(r)$ and initial derivative $f'(r)$. Through the linearlity, we have
	$$
	\begin{aligned}
		f(t)&=f(r)h_2(t-r)+f'(r)h_1(t-r)\\
		&\leq f(r)h_2(t)+f'(r)h_1(t)\\
		&\leq \frac{e^{r_0b_1}-1}{b_1}h_2(t)+e^{r_0b_1}h_1(t),
	\end{aligned}
	$$
	for $t\geq r$.
\end{proof}

\begin{lem}
	By the assumption of \eqref{G}, we have
	\begin{equation}\label{limsup}
		\varlimsup_{r\to\infty}\frac{rh^{(n-1)}_1(r)}{\int_0^r h^{(n-1)}_1\ dt}\leq
		\varlimsup_{r\to\infty}n\frac{rh'_1(r)}{h_1(r)}\leq n\frac{1+\sqrt{1+4B}}{2}.
	\end{equation}
\end{lem}
\begin{proof}
	For $r>0$, we have
	$$
	\begin{aligned}
		h_1^n(r)&=n\int_0^r h^{(n-1)}_1h_1'dt\\
		&\leq nh_1'(r)\int_0^r h^{(n-1)}_1dt,
	\end{aligned}
	$$
	which implies
	$$
	\frac{h^{(n-1)}_1(r)}{\int_0^r h^{(n-1)}_1\ dt}\leq n\frac{h'_1(r)}{h_1(r)}.
	$$
	When the inequality above is equality, we have $h_1'(r)\equiv1$ and
	\begin{equation}\label{brendle2}
		\lambda(t)\equiv0,
	\end{equation}
	which is the case of \cite{Br2}.
	When $B=0$, by Lemma 2.1 in \cite{PRS}, one can get
	$$
	\frac{rh'}{h}\geq 1,\quad r>0.
	$$
	So we have
	$$
	\lim_{r\to\infty}\frac{rh'_1}{h_1}=1+\lim_{r\to\infty}r^2\lambda\frac{h}{rh'}=1=\frac{1+\sqrt{1+4B}}{2}.
	$$
	When $B\neq 0$, by the assumption of \eqref{G}, for any small number $\varepsilon>0$, there exists a large number $R$ such that
	$$
	\lambda(r)\leq \frac{B(1+\varepsilon)}{1+r^2},
	$$
	where $r\geq R>0$. Next, we set $g_1,g_2$ be solutions of the following problems
	$$
	\left\{
	\begin{array}{l}
		g_1''= \frac{B(1+\varepsilon)}{1+r^2}g_1 ,\quad t\in[0,\infty),\\
		\\
		g_1(0)=0,g_1'(0)=1,
	\end{array} \right.
	\quad
	\left\{
	\begin{array}{l}
		g_2''= \frac{B(1+\varepsilon)}{1+r^2}g_2 ,\quad t\in[0,\infty),\\
		\\
		g_2(0)=1,g_2'(0)=0.
	\end{array} \right.
	$$
	It is easy to verify that $g'_1g_2-g_2'g_1\equiv1$. Then we can find a unique solution $(x,y)$ such that
	$$
	\begin{bmatrix}
		h_1(R)\\
		h_1'(R)
	\end{bmatrix}
	=\begin{bmatrix}
		g_1(R) & g_2(R)\\
		g_1'(R) & g_2'(R)
	\end{bmatrix}
	\begin{bmatrix}
		x\\
		y
	\end{bmatrix},
	$$
	which means the initial values of $h(t)$ and $h'(t)$ are the same as $xg_1(t)+yg_2(t)$ and $xg_1'(t)+yg_2'(t)$ respectively. Similar to Lemma 2.5 in \cite{DLL1}, one can obtain
	\begin{equation}\label{ppd}
		\frac{h'_1}{h_1}(r)\leq\frac{xg_1'+yg_2'}{xg_1+yg_2}(r)=
		\Big(\frac{x+y\frac{g_2'}{g_1'}}{x+y\frac{g_2}{g_1}}(r)\Big)
		\frac{g_1'}{g_1}(r),
	\end{equation}
	for $r>R$. By Lemma 2.6 in \cite{DLL1}, we have
	$$
	\lim_{r\to\infty}\frac{g_2}{g_1}=\lim_{r\to\infty}\frac{g_2'}{g_1'}.
	$$
	By Proposition 2.11 in \cite{PRS} and \eqref{ppd}, we have
	$$
	\varlimsup_{r\to\infty}\frac{rh^{(n-1)}_1}{\int_0^r h^{(n-1)}_1\ dt}\leq
	\varlimsup_{r\to\infty}n\frac{rh'_1(r)}{h_1(r)}\leq
	\varlimsup_{r\to\infty}\frac{nrg_1'}{g_1}\leq
	n\frac{1+\sqrt{1+4B
			(1+\varepsilon)}}{2}.
	$$
	Let $\varepsilon\to 0$ and we finish the proof.
\end{proof}

\begin{lem}\label{com-2}
	Let $h$ be the solution of the following problem
	$$
	\left\{
	\begin{aligned}
		&h''=Gh,\quad t\in(0,+\infty),\\
		&h(0)=0,h'(0)=1,
	\end{aligned}
	\right.
	$$
	where $G$ is a nonnegative function satisfing $\varlimsup_{t\to\infty}t^2G(t)<\infty$. Then we have
	$$
	\lim_{t\to\infty}\frac{h(t-c)}{h(t)}=1,
	$$
	for $c\in\mathbb{R}$.
\end{lem}

\begin{proof}
	According to Proposition 2.11 in \cite{PRS}, we have
	$$
	\frac{h'(t)}{h(t)}\le\frac{C_1}{t},
	$$
	for some constant $C_1$.
	Since $h'$ is increasing, one has
	$$
	\frac{h(t)-h(t-c)}{h(t)}\le c\frac{h'(t)}{h(t)}\le\frac{cC_1}{t},\quad t>c,
	$$
	which implies
	$$
	\lim_{t\to\infty}\frac{h(t-c)}{h(t)}=1.
	$$
\end{proof}

\begin{proof}[\textbf{Proof of Theorem \ref{thm1.1}}]
Let $\bar{x}\in A_r$ and $\{e_1,\dots,e_n\}$ be an orthonormal basis of the tangent space $T_{\bar{x}}M$. Constructing the geodesic normal coordinates $(x^1,\dots,x^n)$ around $\bar{x}$, such that $\frac{\partial}{\partial x^i}=e_i$ at $\bar{x}$. Let $\bar{\gamma}(t):=\exp_{\bar{x}}(tD u(\bar{x}))$ for all $t\in[0,r]$. For $1\leq i\leq n$, we define $E_i(t)$ be the parallel transport of $e_i$ along $\bar{\gamma}$. For $1\leq i\leq n$, let $X_i(t)$ be the Jacobi field along $\bar{\gamma}$ with the initial conditions of $X_i(0)=e_i$ and 
$$
\langle D_tX_i(0),e_j\rangle=(D^2u)(e_i,e_j),\quad1\leq j\leq n.
$$
Let $P_{n\times n}(t)$ be a matrix defined by
$$
P_{ij}(t)=\langle X_i(t),E_j(t)\rangle, \quad
1\leq i,j\leq n.
$$
From Lemma \ref{jacovani}, we know $\det P(t)>0,\forall t\in[0,r)$. Obviously, we have $|\det D\Phi_t(\bar{x})|=\det P(t)$. Let $S_{n\times n}(t)$ be a matrix defined by
$$
S_{ij}(t)=R(\bar{\gamma}'(t),E_i(t),\bar{\gamma}'(t),E_j(t)),
\quad
1\leq i,j\leq n,
$$
where $R$ denotes the Riemannian curvature tensor of $M$. By the Jacobi equation, we obtain
$$
\left\{
\begin{array}{l}
	P''(t)=-P(t)S(t),\quad t\in[0,r],\\
	\\
	P_{ij}(0)=\delta_{ij},P_{ij}'(0)=(D^2u)(e_i,e_j).
\end{array} \right.
$$
Let $Q(t)=P(t)^{-1}P'(t),t\in[0,r)$, a simple computation yields
$$
\frac{d}{dt}Q(t)=-S(t)-Q^2(t),
$$
where $Q(t)$ is symmetric. Since $\bar x\in U$, we can set $e_1=\frac{Du(\bar{x})}{|Du(\bar{x})|}$.
In this basis, the first row and column of $S_{n\times n}$ vanishing. Let $\mathrm{tr}_\perp Q=\sum_{i=2}^nQ_{ii}$, we obtain
$$
\left\{
\begin{aligned}
	&Q'_{11}+Q_{11}^2\leq0,\quad t\in[0,r),\\
	&Q_{11}(0)=D^2u(\bar{x})(e_1,e_1),
\end{aligned}
\right.
\quad
\left\{
\begin{aligned}
	&(\mathrm{tr}_\perp Q)'+\frac{1}{n-1}(\mathrm{tr}_\perp Q)^2\leq-\mathrm{tr}S(t)\\
	&\leq (n-1)|D u(\bar{x})|^2\lambda(d(o,\bar{\gamma}(t))),
	\quad t\in[0,r),\\
	&\mathrm{tr}_\perp Q(0)=\Delta u(\bar{x})-D^2u(\bar{x})(e_1,e_1).
\end{aligned}
\right.
$$
By triangle inequality, we get
\begin{equation}\label{triangle}
	d(o,\bar{\gamma}(t))\geq|d(o,\bar{x})-d(\bar{x},\bar{\gamma}(t))|
	=|d(o,\bar{x})-t|D u(\bar{x})||.
\end{equation}
By \eqref{triangle}, we have
$$
(\frac{1}{n-1}\mathrm{tr}_\perp Q)'+(\frac{1}{n-1}\mathrm{tr}_\perp Q)^2
\leq|D u(\bar{x})|^2\lambda(d(o,\bar{\gamma}(t)))
\\
\leq|Du(\bar{x})|^2
\lambda(|d(o,\bar{x})-t|Du(\bar{x})||).
$$
Set $g=\frac{1}{n-1}\mathrm{tr}_\perp Q$, then
$$
\qquad g'(t)+g(t)^2\leq
|D u(\bar{x})|^2
\lambda(|d(o,\bar{x})-t|D u(\bar{x})||)\triangleq\lambda_{\bar{x}}(t).
$$
So we have
$$
\left\{
\begin{array}{l}
	g'(t)+g(t)^2\leq \lambda_{\bar{x}}(t), \quad t\in[0,r),\\
	\\
	g(0)=\frac{1}{n-1}(\Delta u(\bar{x})-Q_{11}(0)).
\end{array} \right.
$$
If we take $\phi_{11}=e^{\int_0^tQ_{11}d\tau}>0$ and $\phi=e^{\int_0^tg(\tau)d\tau}>0$, then 
$$
\left\{
\begin{array}{l}
	\phi_{11}''\leq 0, \quad t\in[0,r),\\
	\\
	\phi_{11}(0)=1,\phi_{11}'(0)=Q_{11}(0),
\end{array} \right.
\quad
\left\{
\begin{array}{l}
	\phi''\leq \lambda_{\bar{x}}(t)\phi, \quad t\in[0,r),\\
	\\
	\phi(0)=1,\phi'(0)=g(0).
\end{array} \right.
$$
Next, $\psi_1,\psi_2$ denote the solutions of the following problems
$$
\left\{
\begin{array}{l}
	\psi_1''= \lambda_{\bar{x}}(t)\psi_1 ,\quad t\in[0,\infty),\\
	\\
	\psi_1(0)=0,\psi_1'(0)=1,
\end{array} \right.
\quad
\left\{
\begin{array}{l}
	\psi_2''= \lambda_{\bar{x}}(t)\psi_2 ,\quad t\in[0,\infty),\\
	\\
	\psi_2(0)=1,\psi_2'(0)=0.
\end{array} \right.
$$
Note that $|Du(\bar{x})|<1$. By Lemma 2.6 in \cite{DLL1}, one can derive
$$
\begin{aligned}
\lim_{r\to\infty}\frac{\psi_2(r)}{\psi_1(r)}=
\lim_{r\to\infty}\frac{\psi'_2(r)}{\psi'_1(r)}&\leq\int_0^\infty\lambda_{\bar{x}}(t)dt
\\
&=\int_0^\infty|Du(\bar{x})|^2\lambda(|d(o,\bar{x})-t|D u(\bar{x})||)dt
\\
&\leq\int_0^\infty\lambda(|d(o,\bar{x})-\tau|)d\tau
\\
&\leq2b_1.
\end{aligned}
$$
Let $\psi(t)=\psi_2(t)+g(0)\psi_1(t)$, using the comparison type result of Lemma 2.5 in \cite{DLL1}, we have
$$
\frac{1}{n-1}\mathrm{tr}_\perp Q(t)=\frac{\phi'}{\phi}\leq\frac{\psi'}{\psi},
$$
and
$$
Q_{11}(t)\leq \frac{\phi'_{11}}{\phi_{11}}\leq
\frac{Q_{11}(0)}{1+tQ_{11}(0)}.
$$
Then we obtain
$$
\frac{d}{dt}\log\det P(t)=\mathrm{tr}Q(t)\leq Q_{11}+(n-1)\frac{\psi'}{\psi}.
$$
Consequently, we get
$$
|\det D\Phi_t(\bar{x})|=\det P(t)\leq(1+tQ_{11}(0)) \psi^{n-1}(t),
$$
for all $t\in[0,r]$. Note that
$$
\begin{aligned}
	0&\leq\phi_{11}(t)\leq1+tQ_{11}(0),\\
	0&\leq\phi(t)\leq\psi(t),
\end{aligned}
$$
for all $t\in[0,r]$. Using the arithmetic-geometric mean inequality, we have
\begin{align*}
|\det D\Phi_r(\bar{x})|&\leq(
\frac{1}{nr}+\frac{n-1}{n}
\frac{\psi_2(r)}{\psi_1(r)}+\frac{1}{n}\Delta u(\bar{x}))^nr\psi^{n-1}_1(r),
\end{align*}
for any $\bar{x}\in A_r$. 
It is easy to show that $|Du(\bar{x})|\psi_1(\frac{t}{|Du(\bar{x})|})$ satisfies \eqref{f}. By \eqref{fineq}, one has
$$
|Du(\bar{x})|\psi_1(\frac{t}{|Du(\bar{x})|})
\leq \frac{e^{r_0b_1}-1}{b_1}h_2(t)+e^{r_0b_1}h_1(t).
$$
Note that for $0<c\leq1$, we can derive that
$$
\begin{aligned}
	(h_1(ct)-ch_1(t))'=c(h_1'(ct)-h_1'(t))\leq0,\\
	(h_2(ct)-ch_2(t))'=c(h_2'(ct)-h_2'(t))\leq0.
\end{aligned}
$$
So we have
$$
\psi_1(t)\leq\frac{e^{r_0b_1}-1}{b_1}\frac{h_2(t|Du|)}{|Du|}+e^{r_0b_1}\frac{h_1(t|Du|)}{|Du|}
\leq\frac{e^{r_0b_1}-1}{b_1}h_2(t)+e^{r_0b_1}h_1(t).
$$
Finally, we get
\begin{equation}\label{D-inte}
\begin{aligned}
&|\{q\in M|d(x,q)<r,\forall x\in\Omega \}\setminus\Omega|\\
&\leq\int_{A_r}|\det D\Phi_r|\\
&\leq\int_{A_r}(
\frac{1}{nr}+\frac{n-1}{n}
\frac{\psi_2(r)}{\psi_1(r)}+\frac{1}{n}\Delta u)^n
(\frac{e^{r_0b_1}-1}{b_1}\frac{h_2(r)}{h_1(r)}+e^{r_0b_1})^
{n-1}rh_1
^{n-1}(r).
\end{aligned}
\end{equation}
Similar to (2.20) in \cite{DLL1}, by Lemma \ref{com-2}, one can obtain
$$
|B_n|\theta=\frac{|\{q\in M|d(x,q)<r,\forall x\in\Omega \}\setminus\Omega|}{n\int_0^rh_1(t)^{n-1}dt}.
$$
Dividing \eqref{D-inte} by $n\int_0^rh_1(t)^{n-1}dt$ and letting $r\to\infty$, using Lemma \ref{hessu} and $\lim_{r\to\infty}\frac{h_2}{h_1}\le b_1$, we can get
$$
|B^n|\theta\leq (2e^{r_0b_1}-1)^{n-1}\int_\Omega f^{\frac{n}{n-1}}
\varlimsup_{r\to\infty}\frac{rh_1
^{n-1}(r)}{n\int_0^rh(t)^{n-1}dt}.
$$
Using \eqref{limsup}, we have
$$
\int_\Omega f^{\frac{n}{n-1}}\geq
\frac{|B^n|\theta}{\frac{1+\sqrt{1+4B}}{2}(2e^{r_0b_1}-1)^{n-1}}.
$$
Using the inequality above, we obtain
$$
\int_{\partial \Omega} f+\int_\Omega |Df|+2(n-1)b_1\int_\Omega f
\geq n\Big(\frac{|B^n|\theta}{\frac{1+\sqrt{1+4B}}{2}(2e^{r_0b_1}-1)^{n-1}}\Big)^\frac{1}{n}
\Big(\int_D f^\frac{n}{n-1}\Big)^\frac{n-1}{n},
$$
for $\Omega$ is connect.
It is easy to show that Theorem \ref{thm1.1} is also true for disconnect case.
\end{proof}

\section{Proof of Theorem \ref{thm1.3}}	
In this section, we assume that the ambient space $M$ is a complete Riemannian manifold of dimension $n+p$ with sectional curvature satisfies \eqref{SG} respect to a base point $o\in M$.  Let $\Sigma\subset M$ be a compact submanifold of dimension $n$ (possibly with boundary $\partial\Sigma$), and $f$ is a positive smooth function on $\Sigma$. Let $B$ be the second fundamental form of $\Sigma$, such that
$$
\langle B(X,Y),V\rangle=\langle \bar{D}_XY,V\rangle,
$$
where $X,Y$ are the tangent vector fields on $\Sigma$, $V$ is a normal vector field along $\Sigma$, $\bar{D}$ is the Levi-Civita connection of $M$, $D$ is the induced connection on $\Sigma$. Without loss of generality, we may assume that $\Omega$ is connected.

By scaling, one can assume that
$$
\int_{\partial\Sigma}f+\int_\Sigma\sqrt{|D f|^2+f^2|H|^2}
+(2nb_1+1)\int_\Sigma f=n\int_\Sigma f^{\frac{n}{n-1}}.
$$
Since $\Sigma$ is connected, we can find a solution of the following Neumann boundary problem

\begin{equation*}
	\left\{
	\begin{array}{lr}
		\text{div}(fD u)=nf^\frac{n}{n-1}-(2nb_1+1)f-\sqrt{|D f|^2+f^2|H|^2},& \text{ in }\Sigma,\\
		\\
		\langle Du,\nu\rangle=1,
		& \text{ on }\partial\Sigma,
	\end{array}
	\right.
\end{equation*}
where $\nu$ is the outward unit normal vector field of $\partial\Sigma$ with respect to $\Sigma$. By standard elliptic regularity theory(see Theorem 6.31 in \cite{GT}), we know that $u\in C^{2,\gamma}$ for each $0<\gamma<1$.

As in \cite{Br2}, we define
\begin{equation*}
\begin{aligned}
U:&=\{x\in\Sigma\setminus\partial\Sigma:|Du(x)|<1\},\\
E:&=\{(x,y):x\in U,y\in T^\perp_x\Sigma,|Du(x)|^2+|y|^2<1\}.
\end{aligned}
\end{equation*}
For each $r>0$, we denote $A_r$ be the set of all points $(\bar{x},\bar{y})\in E$ satisfing
$$
ru(x)+\frac{1}{2}{d}
(x,\exp_{\bar{x}}(rDu(\bar{x}))+r\bar{y})^2\geq
ru(\bar{x})+\frac{1}{2}r^2(|Du(\bar{x})|^2+|\bar{y}|^2)
$$
 for all $x\in\Sigma$. For each $r>0$, we define the transport map $\Phi_r:T^\perp\Sigma\to M$ by
$$
\Phi_r(x,y)=\exp_x(rDu(x)+ry)
$$
for all $x\in\Sigma,y\in T^\perp_x\Sigma$. Using the regularity of $u$, we known that $\Phi_r$ is a $C^{1,\gamma}$ map, $0<\gamma<1$.
\\

Similar to Lemma 4.1 in \cite{Br2}, it is easy to show that
\begin{lem}\label{hess-sub}
Assume that $(x,y)\in E$, then we have
$$
\frac{1}{n}(\Delta_\Sigma u(x)-\langle H(x),y\rangle)\leq f^\frac{1}{n-1}(x)-\frac{2nb_1+1}{n}.
$$
\end{lem}

The following three lemmas are due to Brendle (Lemmas 4.2, 4.3, 4.5 in \cite{Br2}). Their proofs are independent of the curvature assumption of $M$.

\begin{lem}\label{tran-sub}
For each $0\leq\sigma<1$, the set
$$
\{q\in M:\sigma r<d(x,q)<r,\forall x\in\Sigma\}
$$
is contained in the set
$$
\Phi_r(\{(x,y)\in A_r:|Du(x)|^2+|y|^2>\sigma^2\}).
$$
\end{lem}	
\begin{lem}\label{d-mat-sub}
Assume that $(\bar{x},\bar{y})\in A_r$, and let $\bar{\gamma}(t):=\exp_{\bar{x}}(tDu(\bar{x})+t\bar{y})$ for all $t\in[0,r]$. If $Z$ is a smooth vector field along $\bar{\gamma}$ satisfying $Z(0)\in T_{\bar{x}}\Sigma$ and $Z(r)=0$, then
$$
\begin{aligned}
&(D^2u)(Z(0),Z(0))-\langle B(Z(0),Z(0)),\bar{y}\rangle\\
&+\int_0^r\big(|\bar{D}_tZ(t)|^2-\bar{R}(\bar{\gamma}'(t),Z(t),\bar{\gamma}'(t),Z(t))\big)dt\geq0.
\end{aligned}
$$
\end{lem}

\begin{lem}\label{vani-sub}
	Assume that $(\bar{x},\bar{y})\in A_r$, and let $\bar{\gamma}(t):=\exp_{\bar{x}}(tD u(\bar{x})+t\bar{y})$ for all $t\in[0,r]$. Let $\{e_1,\dots,e_n\}$ be an orthonormal basis of $T_{\bar{x}}\Sigma$. Suppose that $W$ is a Jacobi field along $\bar{\gamma}$ satisfying $W(0)\in T_{\bar{x}}\Sigma$ and $\langle\bar{D}_tW(0),e_j\rangle=(D^2u)(W(0),e_j)-\langle B(W(0),e_j),\bar{y}\rangle$ for each $1\leq j\leq n$. If $W(\tau)=0$ for some $\tau\in(0,r)$, then $W$ vanishes identically.
\end{lem}

\begin{proof}[\textbf{proof of Theorem \ref{thm1.3}}]
For any $r>0$ and $(\bar{x},\bar{y})\in A_r$, let $\{e_i\}_{1\leq i\leq n}$ be any given orthonormal basis in $T_{\bar{x}}\Sigma$. Choose a normal coordinate system $(x^1,\cdots,x^n)$ on $\Sigma$ around $\bar{x}$ such that $\frac{\partial}{\partial x^i}=e_i$ at $\bar{x}\ (1\leq i\leq n)$. Let $\{e_\alpha\}_{n+1\leq \alpha\leq n+p}$ be an orthonormal frame field of $T^\perp\Sigma$ around $\bar{x}$ such that $\big(D^\perp e_\alpha\big)_{\bar{x}}=0$ for $n+1\leq\alpha\leq n+p$, where $D^\perp$ denotes the normal connection in the normal bundle $T^\perp\Sigma$ of $\Sigma$. Any normal vector $y$ around $\bar{x}$ can be written as $y=\sum_{\alpha=n+1}^{n+p}y^\alpha e_\alpha$, and thus $(x^1,\cdots,x^n,y^{n+1},\cdots,y^{n+p})$ becomes a local coordinate system on the total space of the normal bundle $T^\perp\Sigma$.

Let $\bar{\gamma}(t):=\exp_{\bar{x}}(tDu(\bar{x})+t\bar{y})$ for all $t\in[0,r]$. For each $1\leq A\leq n+p$, we define $E_A(t)$ be the parallel transport of $e_A$ along $\bar{\gamma}$. For every $1\leq i\leq n$, let $X_i$ be the Jacobi field along $\bar{\gamma}$ with the initial conditions of $X_i(0)=e_i$ and
$$
\begin{aligned}
	&\langle\bar{D}_tX_i(0),e_j\rangle=(D^2u)(e_i,e_j)-\langle B(e_i,e_j),\bar{y}
	\rangle,\quad 1\leq j\leq n,\\
	&\langle\bar{D}_tX_i(0),e_\beta\rangle=
	\langle B(e_i,Du(\bar{x})),e_\beta
	\rangle,\quad n+1\leq \beta\leq n+p.
\end{aligned}
$$
For any $n+1\leq\alpha\leq n+p$, let $X_\alpha$ be the Jacobi field along $\bar{\gamma}$ with the initial conditions of $X_\alpha(0)=0$ and  
$$
\bar{D}_tX_\alpha(0)=e_\alpha.
$$
Using Lemma $\ref{vani-sub}$, we find that $X_A(t)$ are linearly independent for each $t\in(0,r)$.
Let $P_{(n+p)\times (n+p)}(t)$ be a matrix defined by
$$
	P_{AB}(t)=\langle X_A(t),E_B(t)\rangle,\quad t\in[0,r],
$$
where $1\leq A,B\leq n+p$.
Let $S_{(n+p)\times(n+p)}$ be a matrix defined by
$$
S_{AB}(t)=\bar{R}(\bar{\gamma}'(t),E_A(t),\bar{\gamma}'(t),E_B(t)).
$$
Through the Jacobi equation, it is easy to see that
$$
\begin{aligned}
P''(t)&=-P(t)S(t),\quad
P(0)=\begin{bmatrix}
	\delta_{ij}&0\\
	0&0
\end{bmatrix},
\\
P'(0)&=\begin{bmatrix}
	(D^2u)(e_i,e_j)-\langle B(e_i,e_j),\bar{y}
	\rangle&\langle B(e_i,Du(\bar{x})),e_\beta
	\rangle\\
	0&\delta_{\alpha\beta}
	\end{bmatrix}.
\end{aligned}
$$
Let $Q(t)=P(t)^{-1}P'(t),t\in(0,r)$, a simple computation yields
$$
\frac{d}{dt}Q(t)=-S(t)-Q^2(t),
$$
where $Q(t)$ is symmetric. For the matrices $P(t),Q(t)$, we have the following asymptotic expansions (cf. \cite{Br2})
\begin{equation}\label{Q-asy}
	\begin{aligned}
		P(t)&=\begin{bmatrix}
			\delta_{ij}+O(t)&O(t)\\
			O(t) & t\delta_{\alpha\beta}+O(t^2)
		\end{bmatrix},\\
		Q(t)&=\begin{bmatrix}
			D^2u(e_i,e_j)-\langle B(e_i,e_j),\bar{y}
			\rangle+O(t)&O(1)
			\\
			O(1)&\frac{1}{t}\delta_{\alpha\beta}+O(1)
		\end{bmatrix}
	\end{aligned}
\end{equation}
as $t\to0^+$. The triangle inequality implies that
$$
d(o,\bar{\gamma}(t))\geq\big|d(o,\bar{x})-d(\bar{x},\bar{\gamma}(t))\big|
=\big|d(o,\bar{x})-t|D^\Sigma u(\bar{x})+\bar{y}|\big|.
$$
Set
$$
\Lambda_{\bar{x}}(t):=(|D u(\bar{x})|^2+|\bar{y}|^2)
\lambda(\Big|d(o,\bar{x})-t\sqrt{|D u(\bar{x})|^2+|\bar{y}|^2}\Big|).
$$
For $1\le i\le n$, we have
$$
\left\{
\begin{aligned}
	&Q_{ii}'+Q_{ii}^2\leq\Lambda_{\bar x},\quad t\in(0,r),\\
	\\
	&Q_{ii}(0)=D^2u(\bar{x})(e_i,e_i)-\langle B(e_i,e_i),\bar{y}\rangle.
\end{aligned}
\right.
$$
For $n+1\le\alpha\le n+p$, we have
$$
\left\{
\begin{aligned}
	&Q_{\alpha\alpha}'+Q_{\alpha\alpha}^2\leq\Lambda_{\bar x},\quad t\in(0,r),\\
	\\
	&\lim_{t\to0}Q_{\alpha\alpha}(0)=\frac{1}{t}+O(1).
\end{aligned}
\right.
$$
Let $\psi_1,\psi_2$ denote the solutions of the following problems
$$
\left\{
\begin{array}{l}
	\psi_{1}''= \Lambda_{\bar x}\psi_{1} ,\quad t\in[0,\infty),\\
	\\
	\psi_{1}(0)=0,\psi_{1}'(0)=1,
\end{array} \right.
\quad
\left\{
\begin{array}{l}
	\psi_{2}''= \Lambda_{\bar x}\psi_{2} ,\quad t\in[0,\infty),\\
	\\
	\psi_{2}(0)=1,\psi_{2}'(0)=0.
\end{array} \right.
$$
Similarly to (2.6) in \cite{DLL1}, we have
\begin{equation}\label{psi1/2}
\frac{\psi_2}{\psi_1}(r)\le 2b_1\sqrt{|Du(\bar x)|^2+|\bar y|^2}+\frac1r.
\end{equation}
For $1\le i\le n$, we define $\phi_{i}=e^{\int_0^tQ_{ii}d\tau}>0$, then we have
$$
\left\{
\begin{aligned}
	&\phi_{i}''\leq\Lambda_{\bar x}\phi_{i},\quad t\in(0,r),\\
	&\phi_{i}(0)=1,\phi_{i}'(0)=Q_{ii}(0).
\end{aligned}
\right.
$$
Using Lemma 2.5 in \cite{DLL1}, one can get
$$
\frac{\phi'_{i}}{\phi_{i}}(t)\leq\frac{\psi_2'+Q_{ii}(0)\psi_1'}{\psi_2+Q_{ii}(0)\psi_1}(t).
$$
For $n+1\le\alpha\le n+p$, we define  $\phi_\alpha=te^{\int_0^tQ_{\alpha\alpha}-\frac{1}{\tau}~d\tau}$, then we have
$$
\left\{
\begin{aligned}
	&\phi_{\alpha}''\leq\Lambda_{\bar x}\phi_{\alpha},\quad t\in(0,r),\\
	&\phi_{\alpha}(0)=0,\phi_{\alpha}'(0)=1.
\end{aligned}
\right.
$$
By Lemma 2.1 in \cite{PRS}, one can get
$$
\frac{\phi'_\alpha}{\phi_\alpha}(t)\leq\frac{\psi_1'}{\psi_1}(t).
$$
Then we have
$$
\frac{d}{dt}\log\det P(t)=\text{tr}Q(t)\le n\frac{\psi_2'+Q_{ii}(0)\psi_1'}{\psi_2+Q_{ii}(0)\psi_1}(t)+p\frac{\psi_1'}{\psi_1}(t).
$$
Integrating the above inequality, we get
\begin{equation}\label{eq1}
\begin{aligned}
	\det P(1)&\le
	\psi^p_1(1)\Big(\psi_2(1)+Q_{ii}(0)\psi_1(1)\Big)^n
	\lim_{\varepsilon\to0}\frac{\det P(\varepsilon)}{\psi^p_1(\varepsilon)\Big(\psi_2(\varepsilon)+Q_{ii}(0)\psi_1(\varepsilon)\Big)^n}\\
	&=\psi^p_1(1)\Big(\psi_2(1)+Q_{ii}(0)\psi_1(1)\Big)^n.
\end{aligned}
\end{equation}
To get a better estimate on $|\det\bar D\Phi_r(\bar x,\bar y)|$ when $r$ is sufficiently large, we need a special basis. When $|Du(\bar{x})|\neq0,|\bar{y}|\neq0$, we set
$$
\begin{aligned}
	\alpha&= \frac{|Du|}{\sqrt{|Du|^2+|y|^2}},\quad \beta=\frac{|\bar{y}|}{\sqrt{|Du|^2+|y|^2}},
	\\
	{e}_1&=\frac{Du(\bar{x})}{|Du(\bar{x})|},\quad
	{e}_{n+1}=\frac{\bar{y}}{|\bar{y}|}.\\
\end{aligned}
$$
Denote $\bar{e}_1=\alpha e_1+\beta e_{n+1},\bar e_{n+1}=-\beta e_1+\alpha e_{n+1}$. In this basis, we have $S_{11}=0$ and
$$
\begin{aligned}
	\lim_{t\to0}Q_{11}(0)&=\frac{\beta^2}{t}+O(1),\quad
	\lim_{t\to0}Q_{(n+1)(n+1)}(t)=\frac{\alpha^2}{t}+O(1),\\
	\lim_{t\to0}Q_{\alpha\alpha}(t)&=\frac{1}{t}+O(1),
	n+2\le\alpha\le n+p.
\end{aligned}
$$
For $2\leq i\leq n$, we have
$$
\left\{
\begin{aligned}
	&Q'_{11}+Q_{11}^2\leq0,\quad t\in(0,r),\\
	&\lim_{t\to 0}Q_{11}(t)= \frac{\beta^2}{t}+O(1),
\end{aligned}
\right.
\quad
\left\{
\begin{aligned}
	&Q_{ii}'+Q_{ii}^2\leq\Lambda_{\bar x},\quad t\in(0,r),\\
	\\
	&Q_{ii}(0)=D^2u(\bar{x})(e_i,e_i)-\langle B(e_i,e_i),\bar{y}\rangle.
\end{aligned}
\right.
$$
For $n+2\leq\alpha\leq n+p$, we have
$$
\left\{
\begin{aligned}
	&Q'_{(n+1)(n+1)}+Q_{(n+1)(n+1)}^2\leq \Lambda_{\bar x},\quad t\in(0,r),\\
	&\lim_{t\to 0}Q_{(n+1)(n+1)}(t)= \frac{\alpha^2}{t}+O(1),
\end{aligned}
\right.
\quad
\left\{
\begin{aligned}
	&Q_{\alpha\alpha}'+Q_{\alpha\alpha}^2\leq\Lambda_{\bar x},\quad t\in(0,r),\\
	&\lim_{t\to0}Q_{\alpha\alpha}(0)=\frac{1}{t}+O(1).
\end{aligned}
\right.
$$
Similar to the proof of \eqref{eq1}, by Lemma 2.1 in \cite{DLL2} and Lemma 2.5 in \cite{DLL1}, one can get
$$
\frac{d}{dt}\log\det P(t)=\text{tr}Q(t)\le\frac1t+p\frac{\psi_1'}{\psi_1}+\frac{\psi_2'+Q_{ii}(0)\psi_1'}{\psi_2+Q_{ii}(0)\psi_1}.
$$
Set $r>1$, using \eqref{psi1/2} and \eqref{eq1}, we can derive
\begin{equation}\label{eop}
	\begin{aligned}
		|\det\bar D\Phi_r(\bar x,\bar y)|&\le\frac{r\psi_1^{p}(r)\prod_{i=2}^{n}(\psi_2(r)+Q_{ii}(0)\psi_1(r))}{\psi_1^{p}(1)\prod_{i=2}^{n}(\psi_2(1)+Q_{ii}(0)\psi_1(1))}|\det\bar D\Phi_1(\bar x,\bar y)|\\
		&=r\psi_1^p(r)(\psi_2(1)+Q_{11}(0)\psi_1(1))\prod_{i=2}^{n}(\psi_2(r)+Q_{ii}(0)\psi_1(r))\\
		&\le r\psi_1^{n+p-1}(r)\psi_1(1)(2b_1+1+Q_{11}(0))\prod_{i=2}^{n}(2b_1+\frac1r+Q_{ii}(0))\\
		&\le r\psi_1^{n+p-1}(r)\psi_1(1)(\frac{2nb_1+1+\Delta u(\bar x)-\langle H(\bar{x}),\bar{y}\rangle}{n}+\frac{(n-1)}{nr})^n.
	\end{aligned}
\end{equation}
For the case of $|Du(\bar x)|=0,|\bar y|\ne0$ and $|Du(\bar x)|\ne0,|\bar y|=0$, we can also get \eqref{eop} with the similar way. By Lemma \ref{tran-sub} and the formula for change of variables in multiple integrals, for $0\le\sigma<1,r>1$, we find that
\begin{equation*}
	\begin{aligned}
		&\{q\in M:\sigma r<d(x,q)<r,\forall x\in\Sigma\}\\
		&\leq\int_{U}\Big(\int_{\{y\in T^\perp_x\Sigma:
			\sigma^2<|Du(x)|^2+|y|^2<1\}}
		|\det \bar{D}\Phi_r(x,y)|1_{{A_r}(x,y)}dy \Big)d\mathrm{vol}(x)\\
		&\leq\int_{U}\Big(\int_{\{y\in T^\perp_x\Sigma:
			\sigma^2<|Du(x)|^2+|y|^2<1\}}
		r\psi_1^{n+p-1}(r)\psi_1(1)(\frac{2nb_1+1+\Delta u(x)-\langle H(x),y\rangle}{n}+\frac{(n-1)}{nr})^n
		d\mathrm{vol}(x)\\
		&=|B^p|r\psi_1^{n+p-1}(r)\psi_1(1)\int_{\Sigma}
		[(1-|Du(x)|^2)^{\frac{p}{2}}-(\sigma^2-|Du(x)|^2)_+^{\frac{p}{2}})]\cdot(\frac{n-1}{nr}+f^{\frac{1}{n-1}}({x}))^nd\mathrm{vol}(x).
	\end{aligned}
\end{equation*}
Then we have
\begin{equation}\label{x19}
	\begin{aligned}
		&\{q\in M:\sigma r<d(x,q)<r,\forall x\in\Sigma\}\\
		&\leq\frac{p}{2}|B^p|(1-\sigma^2)r\psi_1^{n+p-1}(r)\psi_1(1)
		\int_{\Sigma}(\frac{n-1}{nr}+f^{\frac{1}{n-1}})d\mathrm{vol}(x).
	\end{aligned}
\end{equation}
Similar to the proof of (3.23) in \cite{DLL1}, we have
\begin{equation}\label{x20}
	\begin{aligned}
		&\lim_{r\to+\infty}\frac{|\{p\in M:\sigma r<d(x,p)<r,\forall x\in\Sigma\}|}{(n+p)\int_0^rh_1^{n+p-1}dt}\\
		&=|B^{n+p}|\theta\lim_{r\to+\infty}
		(1-\sigma\frac{h_1^{n+p-1}(\sigma r)}{h_1^{n+p-1}(r)})\\
		&\ge|B^{n+p}|(1-\sigma^{n+p})\theta.
	\end{aligned}
\end{equation}
Dividing \eqref{x19} by ${(n+p)}\int_0^rh_1(t)^{n+p-1}dt$ and sending $r\to\infty$, we have
\begin{equation*}
	\begin{aligned}
		&|B^{n+p}|(1-\sigma^{n+p})\theta\le\lim_{r\to\infty}\frac{|\{p\in M:\sigma r<d(x,p)<r,\forall x\in\Sigma\}|}{(n+p)\int_0^rh_1^{n+p-1}(t)dt}\\
		&\leq \frac{p}{2}|B^p|(1-\sigma^2)\psi_1(1)\int_\Sigma f^{\frac{n}{n-1}}
		\lim_{r\to\infty}\frac{r\psi_1^{n+p-1}(r)}{(n+p)\int_0^rh_1(t)^{n+p-1}dt}\\
		&\le \frac{p}{2}|B^p|(1-\sigma^2)(\frac{e^{2b_1}-1}{2b_1})
		(2e^{r_0b_1}-1)^{n+p-1}\int_\Sigma f^{\frac{n}{n-1}}
		\lim_{r\to\infty}\frac{rh_1^{n+p-1}(r)}{(n+p)\int_0^rh_1(t)^{n+p-1}dt}\\
		&=\frac{p}{2}|B^p|(1-\sigma^2)(\frac{e^{2b_1}-1}{2b_1})
		(2e^{r_0b_1}-1)^{n+p-1}\frac{1+\sqrt{1+4B}}{2}\int_\Sigma
		f^{\frac{n}{n-1}}.
	\end{aligned}
\end{equation*}
Let $\sigma\to1$, one has
\begin{equation}\label{x22}
	(n+p)|B^{n+p}|\theta\leq p|B^p|(\frac{e^{2b_1}-1}{2b_1})
	(2e^{r_0b_1}-1)^{n+p-1}\frac{1+\sqrt{1+4B}}{2}
	\int_\Sigma f
	^{\frac{n}{n-1}}d\mathrm{vol},
\end{equation}
which implies
$$
\begin{aligned}
	&\int_{\partial\Sigma}f+\int_\Sigma\sqrt{|D f|^2+f^2|H|^2}
	+(2nb_1+1)\int_\Sigma f\\
	&\geq
	n\Big(\frac{2b_1(n+p)|B^{n+p}|\theta}{(e^{2b_1}-1)p|B^p|(2e^{r_0b_1}-1)^{n+p-1}\frac{1+\sqrt{1+4B}}{2}}\Big)^{\frac{1}{n}}
	\Big(\int_\Sigma f^{\frac{n}{n-1}}\Big) ^{\frac{n-1}{n}}.
\end{aligned}
$$
\end{proof}

\bibliographystyle{plain}

\bibliography{DLLpoly}

\noindent\mbox{Tian Chong} \\
\mbox{Department of Mathematics,
Shanghai Polytechnic University}\\
\mbox{2360 Jin Hai Road, Pudong District}\\
\mbox{Shanghai, 201209}\\
\mbox{P.R. China}\\
\mbox{\textcolor{blue}{chongtian@sspu.edu.cn}}

\

\noindent\mbox{Han Luo} \\
\mbox{School of Mathematics and Statistics, Nanjing University of Science
and Technology}\\
\mbox{200 Xiaolingwei Street}\\
\mbox{Nanjing, 210094}\\
\mbox{P.R. China}\\
\mbox{\textcolor{blue}{luohan@njust.edu.cn}}

\

\noindent\mbox{Lingen Lu}\\
\mbox{School of Mathematics and Statistics, Hainan University}\\
\mbox{58 Renmin Avenue, Haikou City}\\
\mbox{Hainan, 570228}\\
\mbox{P.R. China}\\
\mbox{\textcolor{blue}{ lulingen@hainanu.edu.cn}}

\end{document}